\documentclass{amsart}

\usepackage{amsmath,amssymb,amsthm,amsfonts,graphicx,color}
\usepackage{geometry}

\usepackage{color, graphics}

\usepackage[english]{babel}

\usepackage[colorlinks=true]{hyperref}

\newcommand{\blue}[1]{{\color{blue} #1}}
\newcommand{\red}[1]{{\color{red} #1}}
\newcommand{\green}[1]{{\color{green} #1}}

\newcommand{\p}{{p}}
\newcommand{\pbf}{{\bf p}}
\newcommand{\h}{{\tt h}}
\newcommand{\q}{{\tt q}}
\newcommand{\pe}{{\tt p}}
\newcommand{\bi}{{\tt b}}
\newcommand{\ta}{{\tt a}}
\newcommand{\x}{{\tt x}}
\newcommand{\y}{{\tt y}}
\newcommand{\s}{{\tt s}}
\newcommand{\z}{{\tt z}}
\newcommand{\zbf}{{\bf z}}
\newcommand{\w}{{\tt w}}
\newcommand{\N}{\mathbb{N}}
\newcommand{\bigB}{\mathbb{B}}
\newcommand{\Z}{\mathbb{Z}}
\newcommand{\Q}{\mathbb{Q}}
\newcommand{\R}{\mathbb{R}}
\newcommand{\C}{\mathbb{C}}
\newcommand{\D}{\delta}
\newcommand{\G}{\mathbb{G}}
\newcommand{\spann}{\mathbb{span}}
\newcommand{\lin}{\mathbb{L}}
\newcommand{\codim}{\text{codim}}
\newcommand{\diver}{\text{div}}
\newcommand{\tr}{\text{tr}}
\newcommand{\ol}{\overline{\lambda}}
\newcommand{\eps}{\varepsilon}
\newcommand{\slk}{\Sigma_{\lambda_{k}}}
\newcommand{\uol}{u_{\overline{\lambda}}}
\newcommand{\sol}{\Sigma_{\overline{\lambda}}}
\newcommand{\ulk}{u_{\lambda_{k}}}
\newcommand{\ul}{\underline{\lambda}}
\newcommand{\uul}{u_{\underline{\lambda}}}
\newcommand{\stul}{\tilde{\Sigma_{\underline{\lambda}}}}
\newcommand{\stlk}{\tilde{\Sigma_{\lambda_{k}}}}
\DeclareMathOperator{\sgn}{sign}
\newcommand{\A}{Area}
\newcommand{\ve}{\varepsilon}
\renewcommand{\theequation}{\thesection.\arabic{equation}}
\theoremstyle{definition}
\newtheorem{definition}{Definition}[section]
\theoremstyle{remark}
\newtheorem{remark}[definition]{Remark}
\theoremstyle{plain}
\newtheorem{lemma}[definition]{Lemma}
\newtheorem{theorem}[definition]{Theorem}
\newtheorem{proposition}[definition]{Proposition}
\newtheorem{corollary}[definition]{Corollary}
\newcommand{\bremark}{\begin{remark} \em}
\newcommand{\eremark}{\end{remark} }

\newcommand{\beq}{\begin{equation}}
\newcommand{\eeq}{\end{equation}}

\def\R {\mathbb{R}}
\def\C {\mathcal{C}}
\def\D {\cD}
\def\N {\mathbb{N}}
\def\S {\mathbb{S}}
\def\Z {\mathbb{Z}}
\def\Sym {{\rm Sym}}
\def\from {\colon}
\def\bv {\mathbf{v}}
\def\bu {\mathbf{u}}
\def\bw {\mathbf{w}}
\def\bW {\mathbf{W}}
\def\bc {\mathbf{c}}
\def\dist{{\rm dist}}
\def\cat{\textrm{cat}}
\def\gen{\textrm{genus}}

\newcommand{\cA}{{\mathcal A}}
\newcommand{\cB}{{\mathcal B}}
\newcommand{\cC}{{\mathcal C}}
\newcommand{\cD}{{\mathcal D}}
\newcommand{\cE}{{\mathcal E}}
\newcommand{\cF}{{\mathcal F}}
\newcommand{\cG}{{\mathcal G}}
\newcommand{\cH}{{\mathcal H}}
\newcommand{\cI}{{\mathcal I}}
\newcommand{\cJ}{{\mathcal J}}
\newcommand{\cK}{{\mathcal K}}
\newcommand{\cL}{{\mathcal L}}
\newcommand{\cM}{{\mathcal M}}
\newcommand{\cN}{{\mathcal N}}
\newcommand{\cO}{{\mathcal O}}
\newcommand{\cP}{{\mathcal P}}
\newcommand{\cQ}{{\mathcal Q}}
\newcommand{\cR}{{\mathcal R}}
\newcommand{\cS}{{\mathcal S}}
\newcommand{\cT}{{\mathcal T}}
\newcommand{\cU}{{\mathcal U}}
\newcommand{\cV}{{\mathcal V}}
\newcommand{\cW}{{\mathcal W}}
\newcommand{\cX}{{\mathcal X}}
\newcommand{\cY}{{\mathcal Y}}
\newcommand{\cZ}{{\mathcal Z}}

\newcommand{\al}{\alpha}
\newcommand{\be}{\beta}
\newcommand{\ga}{\gamma}
\newcommand{\de}{\delta}
\newcommand{\la}{\lambda}
\newcommand{\si}{\sigma}
\newcommand{\om}{\omega}
\newcommand{\De}{\Delta}
\newcommand{\Ga}{\Gamma}
\newcommand{\La}{\Lambda}
\newcommand{\Om}{\Omega}
\newcommand{\Si}{\Sigma}

\newcommand{\Sext}{\mathcal{S}_\mathrm{ext}}
\newcommand{\loc}{\mathrm{loc}}
\newcommand{\Lip}{\mathrm{Lip}}

\renewcommand{\div}{\mathrm{div}}

\newcommand{\problem}[1] {(P)_{#1}}
\newcommand{\weakto}{\rightharpoonup}
\newcommand{\textred}[1]{\textbf{\color{red} #1}}
\newcommand{\ddfrac}[2] {\frac{\displaystyle #1 }{\displaystyle #2} }
\newcommand{\pa}{\partial}
\newcommand{\mf}[1]{\mathbf{#1}}
\newcommand{\bs}[1]{\boldsymbol{#1}}
\newcommand{\wt}{\widetilde}

\newcommand{\bruttateta}{\min\left(1,\frac{2}{N}\right)}

\DeclareMathOperator*{\pv}{pv}
\DeclareMathOperator*{\tsum}{\textstyle{\sum}}
\DeclareMathOperator{\interior}{int\,}
\DeclareMathOperator{\conv}{conv}
\DeclareMathOperator{\supp}{supp}
\DeclareMathOperator*{\osc}{osc}
\DeclareMathOperator*{\meas}{meas}
\DeclareMathOperator{\rad}{rad}

\title[Normalized solutions of mass supercritical Schr\"odinger equations]{Normalized solutions of mass supercritical Schr\"odinger equations with potential}

\author{Thomas Bartsch, Riccardo Molle, Matteo Rizzi and Gianmaria Verzini}

\address{
\hbox{\parbox{5.7in}{\medskip\noindent
Thomas Bartsch\\
Mathematisches Institut, Justus-Liebig-Universit\"at Giessen, \\
Arndtstrasse 2, 35392 Giessen (Germany).\\[2pt]
{\em{E-mail address: }}{\tt Thomas.Bartsch@math.uni-giessen.de.} \\ [5pt]
Riccardo Molle\\
Dipartimento di Roma, Universit\`a di Roma ``Tor Vergata'', \\
Via della Ricerca Scientifica n. 1, 00133 Roma (Italy). \\[2pt]
{\em{E-mail address: }}{\tt molle@mat.uniroma2.it} \\ [5pt]
Matteo Rizzi\\
Mathematisches Institut, Justus-Liebig-Universit\"at Giessen, \\
Arndtstrasse 2, 35392 Giessen (Germany).\\[2pt]
{\em{E-mail address: }}{\tt mrizzi1988@gmail.com.} \\ [5pt]
Gianmaria Verzini\\
Dipartimento di Matematica, Politecnico di Milano, \\
Piazza Leonardo da Vinci, 32, 20133 Milano (Italy). \\[2pt]
{\em{E-mail address: }}{\tt gianmaria.verzini@polimi.it}}}}

\keywords{Nonlinear Schr\"odinger equations, normalized solution, min-max methods.}

\subjclass[2020]{35J50, 35J15, 35J60.}

\thanks{\emph{Acknowledgements.} This work was supported by the MIUR Excellence Department Project CUP E83C18000100006 (Roma Tor Vergata University) and by the INdAM-GNAMPA group. M.R. supported by the Alexander von Humboldt foundation.}

\begin{document}

\begin{abstract}
This paper is concerned with the existence of normalized solutions of the nonlinear Schr\"odinger equation
\[
  -\De u+V(x)u+\la u = |u|^{p-2}u \qquad\text{in $\R^N$}
\]
in the mass supercritical and Sobolev subcritical case $2+\frac4N<p<2^*$. We prove the existence of a solution $(u,\la)\in H^1(\R^N)\times\R^+$ with prescribed $L^2$-norm $\|u\|_2=\rho$ under various conditions on the potential $V:\R^N\to\R$, positive and vanishing at infinity, including potentials with singularities. The proof is based on a new min-max argument.
\end{abstract}

\maketitle

\section{Introduction}
We look for solutions $(u,\la)\in H^1(\R^N)\times\R^+$ of the problem
\begin{equation}\label{main-eq}
\left\{
\begin{aligned}
  -\Delta u+V(x)u+\la u&=|u|^{p-2}u\qquad\text{in $\R^N$}\\
  u\ge0,\ \int_{\R^N}u^2\,dx&=\rho^2
\end{aligned}
\right.
\end{equation}
where $V$ is a fixed potential and $2+\frac{4}{N}<p<2^*=\frac{2N}{N-2}$ ($2^* = +\infty$ if $N=1,2$), i.e.\ $p$ is mass supercritical and Sobolev subcritical. Here $\la$ is a Lagrange multiplier, which appears due to the mass constraint $\|u\|_2=\rho$. Solutions with prescribed $L^2$-norm are known as \textit{normalized solutions}. We consider the case $V\ge0$, $V(x)\to0$ as $|x|\to\infty$ and allow that $V$ has singularities. This equation comes from physics, in fact its solutions are stationary waves for nonlinear Schr\"{o}dinger or Klein-Gordon equations. Recall that solutions of the time-dependent nonlinear Schr\"odinger equation 
\[
  i\pa_t\Phi+\De\Phi-V(x)\Phi+|\Phi|^{p-2}\Phi = 0 \qquad x\in\R^N,\ t>0,
\]
preserve the $L^2$ norm, hence it makes sense to prescribe the mass and not the frequency of standing wave solution $\Phi(x,t)=e^{i\la t}u(x)$. For further physical motivations we refer to \cite{BL}. Normalized solutions of semilinear elliptic equations and systems are also of interest in the framework of ergodic Mean Field Games systems; see \cite{CV,PPVV} and the references therein.

In spite of the importance of normalized solutions from the physical point of view, surprisingly little is known about the existence (or non-existence) of normalized solutions compared with the problem where $\la$ is prescribed instead of $\|u\|_2$. In the case $2<p<2+\frac4N$ one can try to minimize the functional
\begin{equation}\label{main_functional}
  F(u):=\frac{1}{2}\int_{\R^N}\big(|\nabla u|^2+V(x)u^2\big) dx-\frac{1}{p}\int_{\R^N}|u|^{p}dx
\end{equation}
constrained to the $L^2$-sphere
\[
  S_\rho:=\bigg\{u\in H^1(\R^N):\,\int_{\R^N}u^2 dx=\rho^2\bigg\}
\]
in order to solve \eqref{main-eq}. This has been done in the recent paper \cite{IM} by Ikoma and Miyamoto who considered the problem
\[
\left\{
\begin{aligned}
  -\Delta u+V(x)u+\la u&=g(u)\qquad\text{in $\R^N$}\\
  \|u\|_2&=\rho
\end{aligned}
\right.
\]
with $g(s)= o\left(1+|s|^{1+4/N}\right)$ being mass subcritical and satisfying various technical assumptions. Ikoma and Miyamoto found conditions on $V$ so that the functional
\[
  J(u)=\frac12\int_{\R^N}\big(|\nabla u|^2+V(x)u^2\big) dx-\int_{\R^N}G(u)dx
\]
achieves its minimum on $S_\rho$; here $G$ is a primitive of $g$.

In the case $2+\frac{4}{N}<p<2^*$ the functional is unbounded from below (and above) on $S_\rho$. Jeanjean \cite{J} considered the problem
\begin{equation}\label{eq:Jeanjean}
\left\{
\begin{aligned}
  -\Delta u+\la u&=g(u)\qquad\text{in $\R^N$}\\
  \|u\|_2&=\rho
\end{aligned}
\right.
\end{equation}
with $g$ being mass supercritical and Sobolev subcritical. A model linearity is $g(s)=\sum_{j=1}^k |s|^{p_j-2}s$ with $2+\frac4N < p_j < 2^*$ for all $j$. Jeanjean obtained a radial solution $(u,\la)\in H^1_{rad}(\R^N)\times\R^+$ of \eqref{eq:Jeanjean} by a mountain pass argument for $J$ on $S_\rho\cap H^1_{rad}(\R^N)$. In \cite{BdV} the authors obtained the existence of infinitely many solutions of \eqref{eq:Jeanjean} under the same assumptions as in \cite{J}. A major difficulty is that Palais-Smale sequences of $J$ on $S_\rho\cap H^1_{rad}(\R^N)$ need not be bounded, and a bounded Palais-Smale sequence does not need to have a convergent subsequence, even though the embedding $H^1_{rad}(\R^N)\hookrightarrow L^p(\R^N)$ is compact. The main techniques of \cite{J, BdV} depend heavily on $V\equiv0$ and do not work even in the case when $V(x)=V(|x|)$ is radial.

Normalized solutions have also been investigated on a bounded domain $\Om\subset\R^N$.  In \cite{NTV1} the authors consider
\begin{equation}\label{eq:domain}
\left\{
\begin{aligned}
  -\Delta u+\la u&=|u|^{p-2}u&&\qquad\text{in $\Om$}\\
  u&=0 &&\qquad\text{on $\pa\Om$}\\ 
  \|u\|_2&=\rho
\end{aligned}
\right.
\end{equation}
in the case of the unit ball $\Om=B_1(0)\subset\R^N$ and provide necessary and sufficient conditions for the existence of a least energy positive solution. In particular, if $2<p<2+\frac4N$ then \eqref{eq:domain} has a least energy solution for every $\rho>0$, whereas for $2+\frac4N<p<2^*$ there exists $\rho^*>0$ such that \eqref{eq:domain} has a least energy solution only for $0<\rho<\rho^*$. In the latter case there exists even a second positive solution. The mass critical case $p=2+\frac4N$ is somewhat special and has also been treated in \cite{NTV1}. Problem \eqref{eq:domain} on a general bounded domain $\Om\subset\R^N$ has been dealt with in \cite{PV}, considering not only least energy solutions but also solutions of higher Morse index. Also of interest is the recent paper \cite{PPVV} where \eqref{main-eq} and \eqref{eq:domain} have been investigated using Lyapunov-Schmidt type methods for $\rho\to0$ ($\rho\to\infty$ in the mass subcritical case). In the case of \eqref{main-eq} it is assumed that $V$ has a non-degenerate critical point $\xi_0\in\R^N$ and it is proved that the solutions concentrate at $\xi_0$ as $\rho\to0$.

We would also like to mention that recently various results have been obtained for normalized solutions of nonlinear Schr\"odinger systems. We refer to \cite{BJ,BJS,BS1,BS2,BZZ,GJ,IT} for autonomous systems, and to \cite{NTV2,NTV3} for systems with trapping potentials or in bounded domains.

In this paper we do not consider trapping potentials but potentials $V\ge0$ satisfying $V(x)\to0$ as $|x|\to\infty$. A a consequence $F$ has a mountain pass structure but the mountain pass value is the mountain pass value for the case $V\equiv0$ and is not achieved. We present a new linking argument for $F$ constrained to $S_\rho$. This will yield a solution with Morse index $N+1$ if non-degenerate. The argument is inspired by techniques coming from the Sobolev-critical case in unconstrained problems; see \cite{BC2,HMPY,CM,LM}. If $V$ is not radial we cannot work on $H^1_{rad}(\R^N)$ and have to deal with the non-compactness of the embedding $H^1(\R^N)\hookrightarrow L^p(\R^N)$. We also discuss the radial case where a solution is obtained by a mountain pass argument in $S_\rho\cap H^1_{rad}(\R^N)$.

The paper is organized as follows. In the next section we state and discuss our main results. The case of $V\in L^\infty(\R^N)$ will be dealt with in Section \ref{sec-rad}. In Section \ref{sec-pole} we prove our result for potentials having poles. Finally, in Section \ref{sec:rad} we consider the simpler case of radial potentials. Throughout the paper, we will use the notations $\|\cdot\|:=\|\cdot\|_{H^1(\R^N)}$ and $\|\cdot\|_q:=\|\cdot\|_{L^q(\R^N)}$, $q\in[1,\infty]$.

\section{Statement of results}\label{sec:results}
In order to formulate our results we need to introduce some notation. Let $U\in H_{rad}^1(\R^N)$ be the unique solution to
\[
\left\{
\begin{aligned}
 & -\Delta U+U = |U|^{p-2}U\\
 & U>0,\ U(0) = \|U\|_\infty.
\end{aligned}
\right.
\]
Then for $\rho>0$ one obtains a solution $(Z_\rho,\la_\rho)\in H_{rad}^1(\R^N)\times\R^+$ of
\begin{equation}\label{burbuja-rho}
\left\{
\begin{aligned}
  &-\Delta Z_\rho+\la_\rho Z_\rho=|Z_\rho|^{p-2}Z_\rho\\
  &Z_\rho>0,\ \|Z_\rho\|_2=\rho\\
\end{aligned}
\right.
\end{equation}
by scaling:
\[
  Z_\rho(x):=\mu^{-\frac{2}{p-2}}U(x/\mu)
\]
where $\mu>0$ is determined by
\[
  \rho^2=\|Z_\rho\|_2^2=\rho_0^2\mu^{N-\frac{4}{p-2}}, \qquad\qquad\rho_0:=\|U\|_2
\]
and $\la_\rho$ by
\begin{equation}\label{eq:lambda_rho}
  \la_\rho=\mu^{-2}=\left(\frac{\rho}{\rho_0}\right)^{-q-2}>0,\qquad
  q = \frac{4(p-2)}{N(p-2)-4} -2  =\dfrac{4N-2p(N-2)}{N(p - 2) - 4}>0.
\end{equation}
The solution $Z_\rho$ of \eqref{burbuja-rho} is a  critical point, in fact a mountain pass critical point, of
\[
  F_\infty(u):=\frac{1}{2}\int_{\R^N}|\nabla u|^2 dx-\frac{1}{p}\int_{\R^N}|u|^{p}dx
\]
constrained to $S_\rho$. Setting $m_\rho=F_\infty (Z_\rho)$ for arbitrary $\rho>0$, one has $m_{\rho_0}=F_\infty(U)$ and
\begin{equation}\label{eq:poho}
  m_\rho=m_{\rho_0}\left(\frac{\rho}{\rho_0}\right)^{-q} = 
  \dfrac{1}{q}\left(\frac{\rho}{\rho_0}\right)^{-q-2}\,\rho^2 = 
  \dfrac{1}{q}\,\lambda_\rho\rho^2
\end{equation}
(of course, $\rho_0$ depends on $p$ too; see Appendix \ref{app:a} for further details).

Now we can state our basic assumptions on the potential. These are technical and can probably be improved. We state the explicit bounds on the norms in order to make clear that our result is not of perturbation type.
\begin{itemize}
\item[$(V_1)$] $N\ge1$, $V$ and the map $W: x\mapsto V(x)|x|$ are in $L^\infty(\R^N)$, $V\ge0$, $\lim_{|x|\to\infty} V(x)=0$,
\begin{equation}\label{eq:newass}
0<\|V\|_\infty<2\bruttateta\frac{m_\rho}{\rho^2}
\end{equation}
and
\beq\label{eq:V_1}
  \|W\|_\infty \le \frac{m_\rho^{\frac{1}{2}}(2N-p(N-2))}{\rho}\left(\frac{N(p-2)-4}{2(p-2)\big(N(2p-1)(p-2)+2(p(2-N)+2N)\big)}\right)^\frac12.
\eeq
\item[$(V_2)$] $N\ge3$, $V\in L^{\frac{N}{2}}(\R^N)$, $W\in L^N(\R^N)$, $V\ge0$, $\lim_{|x|\to\infty} V(x)=0$,
\begin{equation}\label{eq:newass2}  
\|V\|_{\frac{N}{2}} < \frac{2N(p-2)}{2N-p(N-2)}\left(\left(\frac{N+2}{N}\right)^{\frac{N(p-2)-4}{N(p-2)}}-1\right)
\frac{m_\rho}{\|Z_\rho\|_{2^*}^2},
\end{equation}
and
%
%
\beq\label{eq:V_2}
  2N A^2\left[1+\dfrac{N(p-2)}{2}\right]\|V\|_{\frac{N}{2}}+4A\left[1+\dfrac{N(p-2)^2}{2N-p(N-2)}\right]\|W\|_N\le N(p-2)-4
\eeq
where $A=\frac{1}{\sqrt{\pi N(N-2)}}\left(\frac{\Ga(N)}{\Ga(N/2)}\right)^{\frac{1}{N}}$.
\end{itemize}
Here $\Ga$ denotes the Gamma function. Observe that $A$ is the Aubin-Talenti constant \cite{A,T}, that is the best constant in the Sobolev embedding $H^1(\R^N)\subset L^{2^*}(\R^N)$. Assumption $(V_2)$ allows that the potential has poles, which is important for physical reasons. This assumption can be modified in order to take into account also dimensions $N=1,2$, but for the sake of simplicity here we only deal with the case $N\ge3$ (see Remark \ref{rem:N<3inassV2} ahead). On the other hand, under assumption 
$(V_1)$ we can deal with any dimension.

\begin{theorem}\label{thm:main}
Let $2+\frac4N<p<2^*$ and $\rho>0$. If $V$ satisfies either $(V_1)$ or $(V_2)$ then \eqref{main-eq} has a solution $(u,\la)\in H^1(\R^N)\times\R^+$.
\end{theorem}

\begin{corollary}\label{cor:main}
Let $N\ge1$, $2+\frac4N<p<2^*$ and $V\ge0$ be fixed, with $V,W\in L^\infty(\R^N)$ and 
$\lim_{|x|\to\infty} V(x)=0$. Then \eqref{main-eq} has a solution 
$(u,\la)\in H^1(\R^N)\times\R^+$ provided 
\begin{itemize}
\item either $\rho>0$ is small, depending on $V$, 
\item or  $p$ is close to $2+\frac4N$, depending on $V$, and $\rho$ is small independently of $V$.
\end{itemize}
\end{corollary}

The corollary follows immediately from Theorem \ref{thm:main} and Remark \ref{rem:main} a), b).

\begin{remark}\label{rem:main}
a) By \eqref{eq:poho}, for any fixed $p$,
\[
  \frac{m_\rho}{\rho^2} 
  \to\begin{cases} \infty&\qquad\text{as $\rho\to 0$},\\
                            0     &\qquad\text{as $\rho\to\infty$,}
      \end{cases}
\]
so that for $\rho$ small the $L^\infty$-norms of $V$ and $W$ in $(V_1)$ are allowed to be large. However, we do not have a single $V\not\equiv0$ such that \eqref{main-eq} has a solution for all $\rho>0$. This is an interesting problem, see also point g) below.

b) Again by \eqref{eq:poho}, for any fixed $\rho>0$,
\[
 (N(p-2)-4))\, \frac{m_\rho}{\rho^2}\to
\begin{cases}
0&\qquad\text{if $\rho>\rho_0^*$}\\
+\infty &\qquad\text{if $\rho<\rho_0^*$}
\end{cases}
\qquad\text{as $p\to 2+\frac{4}{N}$}
\]
(from above), where $\rho_0^*$ is $\|U\|_2$ for $p=2+\frac{4}{N}$.
We infer that the $L^\infty$-norms of $V$ and $W$ in $(V_1)$ are allowed to be large when $\rho<\rho_0^*$ and $p$ is close to $2+\frac{4}{N}$.

c)  If $V\in C^{0,\alpha}_{loc}(\R^N)$, the solution found in Theorem \ref{thm:main} is classical and, by the strong maximum principle, it satisfies $u>0$ in $\R^N$. Clearly $u$ is not radial in general since $V$ is not assumed to be radial. If $V$ is radial the assumptions can be weakened, see Theorem \ref{thm:rad} below.

d) The theorem will be proved using variational methods, i.e.\ $u$ will be obtained as a critical point of $F$ constrained to $S_\rho$ and $\la$ will be the Lagrange multiplier. The min-max description of the critical value $F(u)$ implies that the Morse index of $u$ is $N+1$ if $u$ is a non-degenerate critical point of $F|_{S_\rho}$.

e) A related result is contained in \cite{PPVV}, where by means of a perturbation argument a solution with small $\rho$ is obtained, having Morse index $N+1$ and concentrating at a non-degenerate maximum point of $V$. On the other hand, when dealing on bounded domains instead of $\R^N$, it is known that solutions having bounded Morse index can not exist for arbitrarily large $\rho$, see \cite{NTV1,PV}.

f) It is possible to extend our result to the more general problem \eqref{eq:Jeanjean}, when $g$ is for instance a sum of mass-supercritical powers. Since this is just a technical generalization involving suitable assumptions on the growth of $g$ we do not pursue this. On the other hand, combinations of mass-subcritical and mass-supercritical powers are more delicate to treat. Even the autonomous case has been dealt with only very recently in \cite{S1,S2}.

g) The existence of multiple solutions of \eqref{main-eq} is completely open up to now. It is to be expected that a perturbation argument, starting with the infinitely many solutions of \eqref{eq:Jeanjean} obtained in \cite{BdV}, yields the existence of multiple solutions of \eqref{eq:Jeanjean} if $V$ is small but explicit bounds are not known.
\end{remark}

If $V$ is radial the assumptions on $V$ can be weakened. More precisely, the bounds on $\|V\|_\infty$ in $(V_1)$ and on $\|V\|_{\frac{N}{2}}$ in $(V_2)$ can be dropped.
\begin{itemize}
\item[$(V^{rad}_1)$] $N\ge2$, $V$ and the map $W: x\mapsto V(x)|x|$ are in $L^\infty_{rad}(\R^N)$, $V\ge0$, $\lim_{|x|\to\infty} V(x)=0$, and \eqref{eq:V_1} holds.
\item[$(V^{rad}_2)$] $N\ge3$, $V\in L^{\frac{N}{2}}_{rad}(\R^N)$, $W\in L^N(\R^N)$, $V\ge0$, $\lim_{|x|\to\infty} V(x)=0$, and \eqref{eq:V_2} holds.
\end{itemize}
Again, assumption $(V^{rad}_2)$ may be extended to $N=2$ by reasoning as in Remark \ref{rem:N<3inassV2}. On the other hand, both these simplified assumptions can not deal with the case $N=1$, because the embedding $H^1_{rad}(\R)\hookrightarrow L^p(\R)$ is not compact even if $p>2$. 

\begin{theorem}\label{thm:rad}
Let $2+\frac4N<p<2^*$ and $\rho>0$. Then \eqref{main-eq} has a radial solution $(u,\la)\in H^1_{rad}(\R^N)\times\R^+$ if $V$ satisfies $(V^{rad}_1)$ or $(V^{rad}_2)$.
\end{theorem}

\begin{remark}
The solution in Theorem \ref{thm:rad} will be obtained by a mountain pass argument applied to $F$ on $S_\rho\cap H^1_{rad}(\R^N)$, thus it has Morse index $1$ as critical point of $F|_{S_\rho\cap H^1_{rad}(\R^N)}$, if non-degenerate, in the radial space. Considered as a critical point of $F$ on $S_\rho$ in the full space, it will have a larger Morse index in general. If $V$ is radial and radially decreasing then the positive solution $u$ will also be radially decreasing.  It follows that the Morse index of the solution is at least $N+1$ due to the variations coming from the translations. This suggests that it corresponds to the solution from Theorem \ref{thm:main} which is obtained as critical point of $F$ on $S_\rho$ by an $(N+1)$-dimensional linking argument.
\end{remark}

\section{Proof of Theorem \ref{thm:main} in case $(V_1)$}\label{sec-rad}

\subsection{Linking geometry}\label{sec-mpgeometry}
For $h\in\R$ and $u\in H^1(\R^N)$, we introduce the scaling
$$h\star u(x):=e^{\frac{N}{2}h}u(e^h x)$$
which preserves the $L^2$-norm: $\|h\star u\|_2=\|u\|_2$ for all $h\in\R$. For $R>0$ and $h_1<0<h_2$, which will be determined later, we set
$$Q:=B_R\times [h_1,h_2]\subset\R^N\times \R$$
where $B_R=\{x\in\R^N:|x|\le R\}$ is the closed ball of radius $R$ around $0$. For $\rho>0$ we define
$$\Ga_\rho:=\{\ga:Q\to S_\rho\mid\ga\text{ continuous, }\ga(y,h)=h\star Z_\rho(\cdotp-y)\text{ for all $(y,h)\in \pa Q$}\}.$$
We want to find a solution to (\ref{main-eq}) in $S_\rho$ whose energy $F$ is given by
\[
  m_{V,\rho}:=\inf_{\ga\in\Ga_\rho}\max_{(y,h)\in Q}F(\ga(y,h)).
\]
In order to develop a min-max argument, we need to prove that
$$\sup_{\ga\in\Ga_\rho}\max_{(y,h)\in \pa Q} F(\ga(y,h))<m_{V,\rho}$$
at least for some suitable choice of $Q$. For this purpose we prove Propositions \ref{prop-MP-geom}, which gives a lower bound for $m_{V,\rho}$, and \ref{prop-boundary}, which gives an upper bound for $F\circ\ga$ on the boundary $\pa Q$, for any given $\ga\in\Ga_\rho$. The values of $R>0$ and $h_1<0<h_2$ will be determined in Proposition \ref{prop-boundary}. The results up to and including Proposition \ref{prop-MP-geom} hold for arbitrary $R>0$ and $h_1<0<h_2$.\\

In order to study the behaviour of the Palais-Smale sequences, we introduce a suitable splitting Lemma. For $\la>0$ we set
$$I_\la(v):=\frac{1}{2}\int_{\R^N}|\nabla v|^2 dx+\frac{1}{2}\int_{\R^N}V(x)v^2 dx+\frac{\la}{2}\int_{\R^N}v^2 dx-\frac{1}{p}\int_{\R^N}|v|^{p}dx$$
and
$$I_{\infty,\la}(v):=\frac{1}{2}\int_{\R^N}|\nabla v|^2 dx+\frac{\la}{2}\int_{\R^N}v^2 dx-\frac{1}{p}\int_{\R^N}|v|^{p}dx.$$

\begin{lemma}[Splitting Lemma]\label{splitting-lemma}
Let $v_n\in H^1(\R^N)$ be a Palais-Smale sequence for $I_\la$ such that $v_n\to v$ weakly in $H^1(\R^N)$. Then there exist an integer $k\ge 0$, $k$ solutions $w^1,\dots,w^k\in H^1(\R^N)$ to the limit equation
$$-\Delta w+\la w=|w|^{p-2}w$$
and $k$ sequences $\{y_n^j\}_n\subset\R^N$, $1\le j\le k$, such that $|y_n^j|\to\infty$ as $n\to\infty$, and
\begin{equation}\label{splitting-PS}
v_n=v+\sum_{j=1}^k w^j(\cdotp-y^j_n)+o(1)\qquad\text{strongly in $H^1(\R^N)$.}
\end{equation}
Moreover, we have
\begin{equation}
\label{splitting-norm}
\|v_n\|_2^2=\|v\|_2^2+\sum_{j=1}^k \|w^j\|_2^2+o(1)
\end{equation}
and
\begin{equation}\label{splitting-energy}
I_\la(v_n)\to I_\la(v)+\sum_{j=1}^k I_{\infty,\la}(w^j)\qquad\text{as $n\to\infty$.}
\end{equation}
\end{lemma}

The proof of Lemma \ref{splitting-lemma} can be found in \cite[Lemma 3.1]{BC}. The only difference is that \cite{BC} deals with exterior domains, not with $\R^N$, however the proof is exactly the same. We stress that exactly here we need $\la>0$, as in \cite{BC}.\\

Now we recall the notion of \textit{barycentre} of a function $u \in H^1(\R^N)\setminus \{0\}$ which has been introduced in \cite{BW} and in \cite{CP}. Setting
\beq\label{mu}
\nu(u)(x)=\frac{1}{|B_1(0)|}\int_{B_1(x)}|u(y)|dy,
\eeq
we observe that  $\nu(u)$ is bounded and continuous, so  the function
\beq\label{hat}
\widehat{u}(x)=\left[\nu(u)(x)-\frac{1}{2}\max \nu(u)\right]^{+}
\eeq
is well defined, continuous, and has compact support. Therefore we can define
$\beta: H^1(\R^N)\setminus\{0\}\to \R^N$
as
$$
\beta(u)=\frac{1}{\|\widehat{u}\|_{1}}\int_{\R^N}\widehat{u}(x)\, x\, dx.
$$
The map $\beta$ is well defined, because $\widehat{u}$ has compact support, and it is not difficult to verify that it enjoys the following properties:
\begin{itemize}
\item $\beta$ is continuous in $H^1(\R^N)\setminus \{0\}$;
\item if $u$ is a radial function, then $\beta(u)=0$;
\item $\beta(tu)=\beta(u)$ for all $t\neq0$ and for all $u\in H^1(\R^N)\setminus\{0\}$;
\item setting $u_z(x)=u(x-z)$ for $z\in \R^N$ and $u\in H^1(\R^N)\setminus\{0\}$ there holds $\beta(u_z)=\beta(u)+z$.
\end{itemize}
Now we define
\[
\begin{aligned}
&\cD:=\{D\subset S_\rho:\,D\text{ is compact, connected, }h_1\star Z_\rho,\,h_2\star Z_\rho\in D\}\\
&\cD_0:=\{D\in\cD:\,\text{$\beta(u)=0$ for all $u\in D$}\},\\
&\cD_r:=\cD\cap H^1_{rad}(\R^N),
\end{aligned}
\]
and
\[
\begin{aligned}
\ell^r_\rho&:=\inf_{D\in\cD_r}\max_{u\in D} F_\infty(u)\\
\ell^0_\rho&:=\inf_{D\in\cD_0}\max_{u\in D} F_\infty(u)\\
\ell_\rho&:=\inf_{D\in\cD}\max_{u\in D}F_\infty(u).
\end{aligned}
\]
It has been proved in \cite{J} that
\[
  m_\rho = \inf_{\si\in\Si_\rho} \max_{t\in[0,1]} F_\infty(\si(t))
\]
where
\[
  \Si_\rho = \{\si\in\cC([0,1],S_\rho):\si(0)=h_1\star Z_\rho,\,\si(1)=h_2\star Z_\rho\}.
\]


\begin{lemma}\label{Lemma-MP-Geom}
$\ell^r_\rho=\ell^0_\rho=\ell_\rho=m_\rho$
\end{lemma}

\begin{proof}
Clearly $\cD_r\subset\cD_0\subset\cD$, so that $\ell^r_\rho\ge \ell^0_\rho\ge \ell_\rho$. It remains to prove that $\ell_\rho\ge m_\rho$ and $m_\rho\ge\ell^r_\rho$.\\

Arguing by contradiction we assume that $m_\rho>\ell_\rho$. Then $\max_{u\in D}F_\infty(u)<m_\rho$ for some $D\in\cD$, hence $\max_{u\in U_\de(D)}F_\infty(u)<m_\rho$ for some $\de>0$; here $U_\de(D)$ is the $\de$-neighborhood of $D$. Observe that $U_\de(D)$ is open and connected, so it is path-connected. Therefore there exists a path $\si\in\Si_\rho$ such that $\max_{t\in[0,1]}F_\infty(\si(t))<m_\rho$, a contradiction.\\

The inequality $m_\rho\ge\ell^r_\rho$ follows from the fact that the set
$D:=\{h\star Z_\rho:h\in[h_1,h_2]\}\in\cD_r$
satisfies
$$\max_{u\in D}F_\infty(u)=\max_{h\in[h_1,h_2]}F_\infty(h\star Z_\rho)=m_\rho.$$
\end{proof}

The next lemma is a special case of \cite[Theorem 4.5]{G}.

\begin{lemma}\label{lemma-ekeland}
Let $M$ be a Hilbert manifold and let $J\in C^1(M,\R)$ be a given functional. Let $K\subset M$ be compact and consider a subset
$$\cC\subset\{C\subset M:\,C\text{ is compact, $K\subset C$}\}$$
which is homotopy-stable, i.e.\ it is invariant with respect to deformations leaving $K$ fixed. Assume that
$$\max_{u\in K}J(u)<c:=\inf_{C\in\mathcal{C}}\max_{u\in C} J(u)\in\R.$$
Let $\si_n\in \R$ be such that $\si_n\to 0$ and $C_n\in\mathcal{C}$ be a sequence such that
$$0\le\max_{u\in C_n}J(u)-c\le \si_n.$$
Then there exists a sequence $v_n\in M$ such that
\begin{enumerate}
\item $|J(v_n)-c|\le \si_n$,\label{propF}
\item $\|\nabla_M J(v_n)\|\le \tilde{c}\sqrt{\si_n}$,\label{propF'}
\item $\dist(v_n,C_n)\le \tilde{c}\sqrt{\si_n}$,\label{propv_n}
\end{enumerate}
for some constant $\tilde{c}>0$.
\end{lemma}

\begin{lemma}\label{C-m-ineq}
$L_\rho:=\inf_{D\in\cD_0}\max_{u\in D}F(u)>m_\rho$
\end{lemma}

\begin{proof}
Using $V\ge 0$ and Lemma \ref{Lemma-MP-Geom}, we have
\begin{equation}\label{ineq-C}
\max_{u\in D}F(u)\ge \max_{u\in D}F_\infty(u)\ge\ell^0_\rho= m_\rho, \qquad\text{for all $D\in\cD_0$}.
\end{equation}
In order to show that the strict inequality holds in \eqref{ineq-C} we argue by contradiction and assume that there exists a sequence $D_n\in\cD_0$ such that
\begin{equation}
\label{contrad-hp}
\max_{u\in D_n}F(u)\to m_\rho.
\end{equation}
In view of (\ref{ineq-C}), we also have
$$\max_{u\in D_n}F_\infty(u)\to m_\rho$$
Adapting an argument from \cite[Lemma 2.4]{J} we consider the functional
\[
  \wt{F}_\infty: H^1(\R^N)\to\R,\quad \wt{F}_\infty(u,h) := F_\infty(h\star u)
\]
constrained to $M := S_\rho\times\R$. We apply Lemma \ref{lemma-ekeland} with
\[
  K := \{(h_1\star Z_\rho,0),(h_2\star Z_\rho,0)\}
\]
and
\[
  \cC := \{C\subset M: \text{$C$  compact, connected, $K\subset C$}\}
\]
Observe that 
\[
  \wt{\ell}_\rho := \inf_{C\in\cC}\max_{(u,h)\in C}\wt{F}_\infty(u,h) = \ell_\rho = m_\rho
\]
because $\cD\times\{0\}\subset\cC$, hence $\ell_\rho\ge\wt{\ell}_\rho$, and for any $C\in\cC$ we have 
$D:=\{h\star u:(u,h)\in C\}\in\cD$ and
$$
  \max_{(u,h)\in C}\wt{F}_\infty(u,h) = \max_{(u,h)\in C} F_\infty(h\star u) = \max_{v\in D}F_\infty(v)
$$
hence $\ell_\rho\le\wt{\ell}_\rho$. Therefore Lemma \ref{lemma-ekeland} yields a sequence $(u_n,h_n)\in H^1(\R^N)\times\R$ such that
\begin{enumerate}
\item $|\wt{F}_\infty (u_n,h_n)-m_\rho|\to 0$ as $n\to\infty$,
\item $\|\nabla_{S_\rho\times\R}\wt{F}_\infty(u_n,h_n)\|\to 0$ as $n\to\infty$,
\item $\dist((u_n,h_n),D_n\times\{0\})\to 0$ as $n\to\infty$.
\end{enumerate}
In particular, differentiation shows that $v_n:=h_n\star u_n\in H^1(\R^N)$ is a Palais-Smale sequence for $F_\infty$ on $S_\rho$ at level $m_\rho$ satisfying the Pohozaev identity for $F_\infty$, that is
\[
\begin{aligned}
&\frac{1}{2}\|\nabla v_n\|_2^2-\frac{1}{p}\|v_n\|_{p}^{p}\to m_\rho\\
&\|\nabla v_n\|_2^2+\mu_n\|v_n\|_2^2-\|v_n\|_{p}^{p}\to 0 \qquad\text{where }\mu_n:=-\frac{DF_\infty(v_n)[v_n]}{\rho^2}\\
&\|\nabla v_n\|_2^2-\frac{N(p-2)}{2p}\|v_n\|^{p}_{p}\to 0
\end{aligned}
\]
as $n\to\infty$. Moreover, due to point $(3)$ of Lemma \ref{lemma-ekeland} we have $\beta(v_n)\to 0$. As in the proof of \cite[Lemmas 2.4 and 2.5]{J} it follows that $v_n$ is bounded in $H^1(\R^N)$, hence, after passing to a subsequence, $v_n$ converges weakly in $H^1(\R^N)$ to a solution $v\in H^1(\R^N)$ of $-\Delta v+\mu v=|v|^{p-2}v$, and $\mu_n\to\mu>0$. As a consequence of (3) we have $\|v_n\|_2\to\rho$. We claim that 
\begin{equation}\label{conv-H1s}
  v_n\to Z_\rho \qquad\text{strongly in $H^1(\R^N)$.}
\end{equation}
In order to see this we first observe that
$$
\int_{\R^N}\nabla v_n\nabla \varphi \,dx+\int_{\R^N}V(x)v_n \varphi \,dx-\int_{\R^N}|v_n|^{p-2}v_n\varphi \,dx = -\mu_n \int_{\R^N}v_n \varphi \,dx + o(1)\|\varphi\|$$
for every $\varphi\in H^1(\R^N)$, hence $v_n$ is also a Palais-Smale sequence for $I_\mu$ at the level $m_\rho+\frac{\mu}{2}\rho^2$. As a consequence, the splitting Lemma \ref{splitting-lemma} implies
$$v_n=v+\sum_{j=1}^m u_j(\cdotp-y^j_n)+o(1)$$
in $H^1(\R^N)$, where $u_j\ne 0$ are solutions to 
$$-\Delta u_j+\mu u_j=|u_j|^{p-2}u_j$$
and $|y_n^j|\to\infty$. Moreover, setting $\ga:=\|v\|_2$ and $\alpha_j:=\|u_j\|_2$ there holds
$$\rho^2=\gamma^2+\sum_{j=1}^m \alpha_j^2.$$
In addition we have
$$F_\infty(v_n)+\frac{\mu}{2}\rho^2=F_\infty(v)+\frac{\mu}{2}\gamma^2+\sum_{j=1}^m \left(F_\infty(u_j)+\frac{\mu}{2}\alpha_j^2\right)+o(1),$$
which yields that
$$F_\infty(v_n)=F_\infty(v)+\sum_{j=1}^m F_\infty(u_j)+o(1).$$
If $m=0$, then $v_n\to v$ strongly in $H^1(\R^N)$ and \eqref{conv-H1s} follows because $\be(v)=\lim_{n\to\infty}\be(v_n)=0$. If $m=1$ and $v=0$, then $v_n=u_1(\cdotp-y^1_n)+o(1)$, which contradicts the fact that $\beta(v_n)\to 0$. If $m=1$ and $v\ne 0$ then by \eqref{eq:poho}
$$m_\rho+o(1)=F_\infty(v_n)=F_\infty(v)+F_\infty(u_1)+o(1)\ge m_\gamma+m_{\alpha_1} \ge 2m_\rho$$
which is also a contradiction. On the other hand, if $m\ge 2$, then again by \eqref{eq:poho}
$$m_\rho+o(1)=F_\infty(v_n) = F_\infty(v)+\sum_{j=1}^m F_\infty(u_j)+o(1) \ge 2 m_\rho.$$
As a consequence $m=0$ and \eqref{conv-H1s} is true. This implies however
\begin{equation}\notag
F(v_n) = F_{\infty}(v_n)+\frac{1}{2}\int_{\R^N}V(x)v_n^2 dx
\to m_\rho+\frac{1}{2}\int_{\R^N}V(x) Z_\rho^2 dx> m_\rho
\end{equation}
as $n\to\infty$, which contradicts \eqref{contrad-hp}.
\end{proof}

\begin{proposition}
\label{prop-MP-geom}
For any $\rho>0$ there holds $m_{V,\rho}\ge L_\rho$.
\end{proposition}

\begin{proof}
Given a function $\ga:Q=B_R(0)\times[h_1,h_2]\to S_\rho$ and $h\in[h_1,h_2]$ we consider the mapping
$$f_h:B_{R}\to\R^N,\quad y\mapsto \beta\circ\ga (y,h) .$$
We note that $f_{h_i}(y)=0$ if and only if $y=0$, for $i=1,2$, and $f_h(y)=y\ne 0$ for any $y\in \pa B_R$, so that
$\deg(f_h,B_R,0)=1$ for all $h\in[h_1,h_2]$. Therefore, by the degree theory, there exists a connected compact set $Q_0\subset Q$ such that $(0,h_i)\in Q_0$ for $i=1,2$, and $\beta\circ\ga(y,h)=0$, for any $(y,h)\in Q_0$. Hence the set $D_0:=\ga(Q_0)\in \cD_0$ satisfies
$$\max_{(y,h)\in Q} F(\ga(y,h))\ge\max_{u\in D_0}F(u)$$
which concludes the proof.
\end{proof}

\begin{proposition}
\label{prop-boundary}
For any $\rho>0$ and for any $\eps>0$ there exist $\bar{R}>0$ and $\bar{h}_1<0<\bar{h}_2$ such that for $Q=B_R\times[h_1,h_2]$ with $R\ge\bar{R}$, $h_1\le\bar{h}_1$, $h_2\ge\bar{h}_2$ the following holds:
\[
  \max_{(y,h)\in\pa Q} F(h\star Z_\rho(\cdotp-y))<m_\rho+\eps.
\]
\end{proposition}

\begin{proof}
We have
$$F(h\star Z_\rho(\cdotp-y))=F_\infty(h\star Z_\rho)+\frac{e^{hN}}{2}\int_{\R^N} V(x)Z_\rho(e^h(x-y))^2 dx$$
and
\[
\begin{aligned}
  F_\infty(h\star Z_\rho) &=\frac{e^{2h}}{2}\int_{\R^N}|\nabla Z_\rho|^2 dx-\frac{e^{\frac{N}{2}(p-2)h}}{p}\int_{\R^N} Z_\rho^{p}dx\\
      &=\begin{cases}
             O(-e^{\frac{N}{2}(p-2)h})\to-\infty &\quad\text{as $h\to\infty$}\\
             O(e^{2h})\to 0 &\quad\text{as $h\to-\infty$.}
           \end{cases}
\end{aligned}
\]
Moreover, there holds
\[
\begin{aligned}
e^{hN}\int_{\R^N} V(x)Z_\rho(e^h(x-y))^2 dx&=\int_{\R^N}V(e^{-h}x+y)Z_\rho^2(x) dx\\
&\begin{cases}
\to 0&\quad \text{as $h\to-\infty$ uniformly in $y\in \R^N$}\\
\le\|V\|_\infty\rho^2&\quad\text{for all $h\in\R$}
\end{cases}
\end{aligned}
\]
because $V(x)\to 0$ as $|x|\to\infty$. As a consequence we deduce
$$ F(h\star Z_\rho(\cdotp-y))\to-\infty \qquad\text{as $h\to\infty$}$$
uniformly in $y\in B_R$, any $R>0$, and
$$F(h\star Z_\rho(\cdotp-y))\to 0\qquad\text{as $h\to-\infty$}$$
uniformly in $y\in\R^N$. Therefore, we have
\begin{equation}\notag
\max_{y\in B_R,h\in\{h_1,\, h_2\}} F(h\star Z_\rho(\cdotp-y))<m_\rho
\end{equation}
provided $h_1<0$ is small enough and $h_2>0$ is large enough. Moreover, for $|y|=R$ large enough and $h\in[h_1,\, h_2]$, we choose $\alpha\in(0,1)$ such that $\alpha(1+e^{-h_1})<1$ so that we have
\begin{equation}\notag
\begin{aligned}
&e^{Nh}\int_{\R^N} V(x)Z_\rho^2(e^h(x-y)) dx\\\
&\hspace{1cm}
  \le e^{Nh}\int_{|x|>\alpha R}V(x)Z_\rho^2(e^h(x-y)) dx+e^{Nh}\int_{|x-y|>\alpha R e^{-h}}V(x)Z_\rho^2(e^h(x-y))^2 dx.
\end{aligned}
\end{equation}
The first integral is bounded by
$$e^{Nh}\int_{|x|>\alpha R}V(x)Z_\rho^2(e^h(x-y)) dx\le \|V\|_{L^\infty(|x|>\alpha R)}\int_{\R^N}Z_\rho^2 dx\to 0$$
as $R\to\infty$ and
\begin{equation}\notag
\begin{aligned}
e^{Nh}\int_{|x-y|>\alpha R e^{-h}}|V(x)|Z_\rho(e^h(x-y))^2 dx&=\int_{|\xi|>\alpha R}V(y+e^{-h}\xi)Z_\rho^2(\xi)d\xi\\
&\le\|V\|_\infty\int_{|\xi|>\alpha R}Z_\rho^2 dx\to 0
\end{aligned}
\end{equation}
as $R\to\infty$, which concludes the proof.
\end{proof}

By Propositions \ref{prop-boundary} and \ref{prop-MP-geom} we may choose $R>0$ and $h_1<0<h_2$ such that
\begin{equation}\label{mountain-pass-geom}
\max_{(y,h)\in\pa Q} F(h\star Z_\rho(\cdotp-y)) < m_{V,\rho}.
\end{equation}
This implies that $F$ has a linking geometry and that there exists a Palais-Smale sequence at the level $m_{V,\rho}$. The aim of the next Sections will be to prove that $m_{V,\rho}$ is a critical value for $F$. Moreover, for future purposes, we prove the following Lemma.

\begin{lemma}
\label{lemma-upper-m-eps}
If $|h_1|,\,h_2$ are large enough, then
$$m_{V,\rho}\le m_\rho+\frac{\|V\|_\infty}{2}\rho^2.$$
\end{lemma}

\begin{proof}
This follows from
\[
\begin{aligned}
  m_{V,\rho}&\le \max_{(y,h)\in Q}\left\{F_\infty(h\star Z_\rho(\cdotp-y))+\frac{1}{2}\int_{\R^N} V(x)(h\star Z_\rho)^2(x-y)dx\right\}\\
                  &\le m_\rho+\frac{1}{2}\|V\|_\infty\rho^2
\end{aligned}
\]
provided $|h_1|,\,h_2$ are large enough.
\end{proof}

\begin{remark}
In particular, under assumption $(V_1)$, Lemma \ref{lemma-upper-m-eps} yields that $m_{V,\rho}< 2m_\rho$.
\label{rem-m_Vrho}
\end{remark}

\subsection{A bounded Palais-Smale sequence}\label{sec-bd-PS-seq}
The aim of this section is to construct a bounded Palais-Smale sequence $v_{n}\in H^1(\R^N)$ of $F$ at the level $m_{V,\rho}$. Adapting the approach from \cite{J} we introduce the functional
\begin{equation}
  \wt{F}(u,h):=F(h\star u),\qquad\text{for all $(u,h)\in H^1(\R^N)\times \R$.}
\end{equation}
and define
$$\wt{\Ga}_\rho:=\{\wt{\ga}:Q\to S_\rho\times\R:\,\wt{\ga}(y,h):=(h\star Z_\rho(\cdotp-y),0),\,\forall\,(y,h)\in\pa Q\}$$
and
$$\wt{m}_{V,\rho}:=\inf_{\wt{\ga}\in\wt{\Ga}_\rho}\max_{(y,h)\in Q}\wt{F}(\wt{\ga}(y,h)).$$

\begin{lemma}
a) $\wt{m}_{V,\rho}=m_{V,\rho}.$

b) If $(u_n,h_n)_n$, is a (PS)$_c$ sequence for $\wt{F}$ then $(h_n\star u_n)_n$, is a (PS)$_c$ sequence for $F$. Conversely, if $(u_n)_n$ is a (PS)$_c$ sequence for $F$ then $(u_n,0)_n$, is a (PS)$_c$ sequence for $\wt{F}$.
\end{lemma}

\begin{proof}
a) Since $\Ga_\rho\times\{0\}\subset\wt{\Ga}_\rho$, then $m_{V,\rho}\ge \wt{m}_{V,\rho}$. On the other hand, for any $\wt{\ga}=(u,h)\in\wt{\Ga}_\rho$, the function $\ga:=h\star u\in\Ga_\rho$ satisfies
$$\max_{(y,t)\in Q}\wt{F}(\wt{\ga}(y,t))=\max_{(y,h)\in Q}F(\ga(y,t))$$
so that $m_{V,\rho}\le \wt{m}_{V,\rho}$.\\

b) This has been proved in \cite{J}.
\end{proof}

\begin{proposition}
Let $\wt{g}_n\in\wt{\Ga}_\rho$ be a sequence such that
$$\max_{(y,h)\in Q}\wt{F}(\wt{g}_n(y,h))\le m_{V,\rho}+\frac{1}{n}.$$
Then there exist a sequence $(u_{n},h_{n})\in S_\rho\times\R$ and $\tilde{c}>0$ such that
$$m_{V,\rho}-\frac{1}{n} \le \wt{F}(u_{n},h_{n}) \le m_{V,\rho}+\frac{1}{n}$$
$$\min_{(y,h)\in Q}\|(u_{n},h_{n})-\wt{g}_n(y,h)\|_{H^1(\R^N)\times\R}\le\frac{\tilde{c}}{\sqrt{n}}$$
$$\|\nabla_{S_\rho\times\R}\wt{F}(u_{n},h_{n})\|\le\frac{\tilde{c}}{\sqrt{n}}.$$
\label{prop-Ekeland}
\end{proposition}

The last inequality means:
$$\big|D\wt{F}(u_{n},h_{n})[(z,s)]\big|\le\frac{\tilde{c}}{\sqrt{n}}(\|z\|_{H^1(\R^N)}+|s|)$$
for all
$$(z,s)\in \bigg\{(z,s)\in H^1(\R^N)\times\R:\,\int_{\R^N}zu_{n}=0\bigg\}.$$

\begin{proof}
This follows from Lemma \ref{lemma-ekeland} applied to $\wt{F}$ with
\begin{equation}\notag
\begin{aligned}
&M:=S_\rho\times\R,\qquad K:=\{(h\star Z_\rho(\cdotp-y),0):\,(y,h)\in \pa Q\}\\
&\mathcal{C}=\wt{\Ga}_\rho,\qquad
C_n:=\{\wt{g}_n(y,h):(y,h)\in Q\}.
\end{aligned}
\end{equation}
\end{proof}

As a consequence we obtain a bounded Palais-Smale sequence for $F$ at the level $m_{V,\rho}$.

\begin{proposition}\label{prop:bddPSseq}
There exists a \textbf{bounded} sequence $(v_{n})_n$ in $S_\rho$ such that
\begin{equation}\label{Palais-Smale-meps}
  F(v_n)\to m_{V,\rho},\qquad \nabla_{S_\rho}F(v_n)\to 0
\end{equation}
and
\begin{equation}
\label{Pohozaev-id}
\|\nabla v_{n}\|_2^2-\frac{N(p-2)}{2p}\|v_{n}\|_{p}^{p}+\frac{1}{2}\int_{\R^N}V(x)(N v_{n}^2+2v_n\nabla v_n\cdotp x)dx\to 0
\end{equation}
as $n\to\infty$. Moreover, the sequence of Lagrange multipliers
\begin{equation}
\label{def-lambda_n}
\la_n:=-\frac{DF(v_n)[v_n]}{\rho^2}
\end{equation}
admits a subsequence $\la_n\to\la$, with
\begin{equation}\label{eq:lafromab}
0< \lambda  \le \left(2q + \frac{8}{N(p-2)-4} \min\left(1,\frac{2}{N}\right)\right)\,\frac{m_\rho}{\rho^2}.
\end{equation}
\label{prop-bounded-PS}
\end{proposition}

\begin{proof}
First we choose a sequence $g_n\in\Ga$ such that
$$\max_{(y,h)\in Q}F(g_n(y,h))\le m_{V,\rho}+\frac{1}{n}.$$
Since $F(u)=F(|u|)$, for any $u\in H^1(\R^N)$, we can assume that $g_n(y,h)\ge 0$ almost everywhere in $\R^N$. Applying Proposition \ref{prop-Ekeland} to $\wt{g}_n(y,h):=(g_n(y,h),0)\in\wt{\Ga}_\rho$, we can prove the existence of a sequence $(u_n,h_n)\in H^1(\R^N)\times \R$ such that $F(h_n\star u_n)\to m_{V,\rho}$. Note that, by Proposition \ref{prop-Ekeland}, we have
$$\min_{(y,h)\in Q}\|(u_{n},h_{n})-\wt{g}_n(y,h)\|_{H^1(\R^N)\times\R}\le\frac{\tilde{c}}{\sqrt{n}}$$
so that $h_n\to 0$ as $n\to\infty$ and there exists $(y_n,\bar{h}_n)\in B_R\times[h_1,h_2]$ such that
\begin{equation}
\label{v>=0}
v_n:=h_{n}\star u_{n}=h_{n}\star u_{n}-g_n(y_n,\bar{h}_n)+g_n(y_n,\bar{h}_n)=g_n(y_n,\bar{h}_n)+o(1)
\end{equation}
as $n\to \infty$, with $g_n(y_n,\bar{h}_n)\ge 0$ a. e. in $\R^N$. Moreovoer, $(v_n)_n$ is a Palais-Smale sequence, again by Proposition \ref{prop-Ekeland}.
Similarly we have for $w\in H^1(\R^N)$, setting $\wt{w}_n:=(-h_n)\star w$, that
$$D (F-F_\infty)(v_{n})[w]=\int_{\R^N}V(e^{-h_{n}} x)u_{n}\wt{w}_n,$$
which implies
$$DF(v_{n})[w] =D\wt{F}(u_{n},h_{n})[(\wt{w}_n,0)]+o(1)\|\wt{w}_n\|.$$
Moreover, it is easy to see that $\int_{\R^N}v_{n}w=0$ is equivalent to $\int_{\R^N}u_{n}\wt{w}=0$. Since $\|\wt{w}_n\|^2_{H^1(\R^N)}\le 2\|w\|^2_{H^1(\R^N)}$ for $n$ large, we have \eqref{Palais-Smale-meps}.\\

Now we prove (\ref{Pohozaev-id}). Using
$$D\wt{F}(u_{n},h_{n})[(0,1)]\to 0,\quad h_n\to0,\qquad\text{as $n\to\infty$}$$
computing
\[
\begin{aligned}
&\pa_h\left(\int_{\R^N}V(x)e^{Nh}u^2(e^{h} x)dx\right)=\int_{\R^N}V(x)(N e^{Nh}u^2(e^h x)+2e^{Nh}u(e^h x)\nabla u(e^h x)\cdotp e^h x) dx
\end{aligned}
\]
and
$$\pa_h F_\infty(h\star u)=\|\nabla v\|_2^2-\frac{N(p-2)}{2p}\|v\|_{p}^{p},$$
we obtain
\[
\|\nabla v_{n}\|_2^2-\frac{N(p-2)}{2p}\|v_{n}\|_{p}^{p}+\frac12\int_{\R^N}V(x)(Nv_n^2+2v_n\nabla v_n\cdotp x) dx\to 0,\qquad\text{as $n\to\infty$}
\]
i.e.\ $v_n$ almost satisfies the Pohozaev identity.\\

In order to see that $v_n$ is bounded we set
\beq\label{eq:def-a_n}
a_n:=\|\nabla v_n\|_2^2,\qquad b_n:=\|v_n\|_{p}^{p},\qquad c_n:=\int_{\R^N}V(x)v_n^2 dx,\qquad d_n:=\int_{\R^N}V(x)v_n\nabla v_n\cdotp x dx
\eeq
relations (\ref{Palais-Smale-meps}), (\ref{Pohozaev-id}) and (\ref{def-lambda_n}) can be expressed in the form
\begin{eqnarray}\label{PS-v_n}
a_n+c_n-\frac{2}{p}b_n= 2m_{V,\rho}+o(1)\qquad\text{as $n\to\infty$}
\end{eqnarray}
\begin{eqnarray}\label{Pohozaev-new}
a_n-\frac{N(p-2)}{2p}b_n+\frac{N}{2}c_n+d_n=o(1)\qquad\text{as $n\to\infty$}
\end{eqnarray}
\begin{eqnarray}\label{lambda_n-new}
a_n+c_n+\la_n \rho^2= b_n+o(1)\left(a_n^{\frac12}+1\right)\qquad\text{as $n\to\infty$}
\end{eqnarray}
so that
\[
\frac{N(p-2)-4}{2p}b_n= 2m_{V,\rho}+\frac{N-2}{2}c_n+d_n+o(1).
\]
As a consequence
\begin{equation}\label{a_n-bd}
\begin{aligned}
a_n&=\frac{4}{N(p-2)-4}\left(2m_{V,\rho}+\frac{N-2}{2}c_n+d_n\right)-c_n+2m_{V,\rho}+o(1)\\
&=\frac{N(p-2)}{N(p-2)-4}2m_{V,\rho}+\frac{N(4-p)}{N(p-2)-4}c_n+\frac{4}{N(p-2)-4}d_n+o(1)
\end{aligned}
\end{equation}
so that, using the inequalities
$$c_n\le \|V\|_\infty\rho^2< 2m_\rho,\qquad m_{V,\rho}< 2m_\rho$$
we have with $W(x)=V(x)|x|$:
\[
\begin{aligned}
0&\le (N(p-2)-4)a_n= 2N(p-2)m_{V,\rho}+(4-p)Nc_n+4d_n+o(1)\\
&\le 4N(p-2)m_\rho+6N m_\rho+4\|W\|_\infty\rho a_n^{\frac{1}{2}}+o(1)
\end{aligned}
\]
or equivalently
$$(N(p-2)-4)a_n-4\|W\|_\infty\rho a_n^{\frac{1}{2}}-2N(2p-1)m_\rho+o(1)\le 0.$$
This implies
\begin{equation}\label{bound-a_n}
(N(p-2)-4)a_n^{\frac{1}{2}}\le 2\|W\|_\infty\rho+\sqrt{4\|W\|_\infty^2\rho^2+2N(2p-1)m_\rho(N(p-2)-4)}+o(1),
\end{equation}
so that $a_n$ is bounded.\\

To conclude, we prove that the sequence $\la_n$ admits a subsequence that converges to a positive bounded limit. Since $a_n$ and $c_n$ are bounded, then also $d_n$ is bounded, due to the H\"{o}lder inequality, therefore also $b_n$ and $\la_n$ are. Up to a subsequence, we can assume that
$$a_n\to a\ge 0,\qquad b_n\to b\ge 0,\qquad c_n\to c\ge 0,\qquad d_n\to d\in\R,\qquad\la_n\to\la\in\R.$$
Passing to the limit in (\ref{PS-v_n}), (\ref{Pohozaev-new}) and (\ref{lambda_n-new}), we have
\[
a+c-\frac{2}{p}b= 2m_{V,\rho}
\]
\[
a-\frac{N(p-2)}{2p}b+\frac{N}{2}c+d=0
\]
\[
a+c+\la \rho^2= b
\]
which imply
\begin{equation}\label{lambda>0}
\begin{aligned}
\la\rho^2&=\frac{p-2}{p}b-2m_{V,\rho}\\
&=\frac{2(p-2)}{N(p-2)-4}\left(2m_{V,\rho}+\frac{N-2}{2}c+d\right)-2m_{V,\rho}\\
&=\frac{p(2-N)+2N}{N(p-2)-4}2m_{V,\rho}+\frac{(N-2)(p-2)}{N(p-2)-4}c+\frac{2(p-2)}{N(p-2)-4}d\\
&>\frac{p(2-N)+2N}{N(p-2)-4}2m_{\rho}+\frac{(N-2)(p-2)}{N(p-2)-4}c+\frac{2(p-2)}{N(p-2)-4}d
\end{aligned}
\end{equation}
so that $\la>0$ provided
$$(p-2)|d|\le m_\rho(2N-p(N-2)).$$
Using \eqref{bound-a_n} it is possible to see that
\begin{equation}\notag
\begin{aligned}
(p-2)|d|&\le (p-2)\|W\|_\infty\rho a^{\frac{1}{2}}\\
&\le(p-2)\|W\|_\infty\rho \frac{2\|W\|_\infty\rho+\sqrt{4\|W\|_\infty^2\rho^2+2N(2p-1)m_\rho(N(p-2)-4)}}{N(p-2)-4}\\
&\le m_\rho(p(2-N)+2N)
\end{aligned}
\end{equation}
provided
$$\|W\|_\infty\le\frac{m_\rho^{\frac{1}{2}}}{\rho}
\left(\frac{(N(p-2)-4)(p(2-N)+2N)^2}{2(p-2)\big(N(2p-1)(p-2)+2(p(2-N)+2N)\big)}\right)^{\frac12}.$$
In the same way, since 
\[
\begin{split}
\la\rho^2&=\frac{2N-p(N-2)}{N(p-2)-4}2m_{V,\rho}+\frac{(N-2)(p-2)}{N(p-2)-4}c+\frac{2(p-2)}{N(p-2)-4}d,\\
(p-2)|d|&\le m_\rho(2N-p(N-2))\\
c&\le \|V\|_\infty\rho^2,
\end{split}
\]
using \eqref{eq:newass} and $m_{V,\rho} \le m_\rho + \frac12\|V\|_\infty\rho^2$ we deduce
($\theta = \bruttateta$) 
\[
\begin{split}
\la\rho^2&\le\left((4 + 2\theta)\frac{2N-p(N-2)}{N(p-2)-4}+ 2\theta\frac{(N-2)(p-2)}{N(p-2)-4}\right) m_\rho\\
&=\left((4 + 2\theta)\frac{2N-p(N-2)}{N(p-2)-4} - 2\theta\frac{2N-p(N-2)-4}{N(p-2)-4}\right) m_\rho\\
&= \left(2q + \frac{8\theta}{N(p-2)-4} \right)\,m_\rho.
\end{split}
\]
\end{proof}

\subsection{Convergence of the Palais-Smale sequence}\label{sec-convergence-PS}
Since $v_{n}$ is bounded in $H^1(\R^N)$, after passing to a subsequence it converges weakly in $H^1(\R^N)$ to 
$v\in H^1(\R^N)$. Due to (\ref{v>=0}) we have $g_n(y_n,\bar{h}_n)\to v$ weakly in $H^1(\R^N)$ and pointwise a.e.\ in $\R^N$, so that $v\ge 0$, since $g_n(y_n,\bar{h}_n)\ge 0$. Moreover, by weak convergence, $v$ is a weak solution  of
\begin{equation}\label{eq:v}
-\Delta v+(\la+V)v=|v|^{p-2}v
\end{equation}
such that $\|v\|_2\le \rho$. It remains to prove that $\|v\|_2=\rho$.\\

We note that, since
$$
\int_{\R^N}\nabla v_n\nabla \varphi \,dx+\int_{\R^N}V(x)v_n \varphi \,dx-\int_{\R^N}|v_n|^{p-2}v_n\varphi \,dx = -\la_n \int_{\R^N}v_n \varphi \,dx + o(1)\|\varphi\|$$
for every $\varphi\in H^1(\R^N)$, $v_n$ is also a Palais-Smale sequence for $I_\la$ at level $m_{V,\rho}+\frac{\la}{2}\rho^2$, therefore, by the splitting Lemma \ref{splitting-lemma}, we have
$$v_n=v+\sum_{j=1}^k w^j(\cdotp-x^j_n)+o(1),$$
being $w^j$ solutions to 
$$-\Delta w^j+\lambda w^j=|w^j|^{p-2}w^j$$
and $|x^j_n|\to\infty$. We note that, if $k=0$, then $v_n\to v$ strongly in $H^1(\R^N)$, hence $\|v\|_2=\rho$ and we are done, thus we can assume that $k\ge 1$, or equivalently $\ga:=\|v\|_2<\rho$.\\

First we exclude the case $v=0$. In fact, if $v=0$ and $k=1$, we would have $w^1>0$ and $\|w^1\|_2=\rho$ so that \eqref{splitting-energy} would give $m_{V,\rho}=m_\rho$, which is not possible due to Proposition \ref{prop-MP-geom}. On the other hand, if $k\ge 2$, using $F_\infty(w^j)\ge m_{\alpha_j}$
and
\begin{equation}\label{monotonicity-m_rho}
  m_{\alpha}>m_{\beta} \qquad\text{if $\alpha<\beta$}
\end{equation}
the condition $F(v_n)\to m_{V,\rho}$ and (\ref{splitting-energy}) would give
$$2 m_\rho\le m_{V,\rho}\le m_\rho+\frac{\|V\|_\infty\rho^2}{2},$$
which contradicts assumption $(V_1)$.\\

Therefore from now on we will assume that $v\ne 0$. From $F(v_n)\to m_{V,\rho}$ we deduce
\[
m_{V,\rho}+\frac{\la}{2}\rho^2=F(v)+\frac{\la}{2}\ga^2+\sum_{j=1}^k F_\infty(w^j)+\frac{\la}{2}\sum_{j=1}^k \alpha_j^2
\]
where $\alpha_j:=\|w_j\|_2^2$. Using $F_\infty(w^j)\ge m_{\alpha_j}$, \eqref{monotonicity-m_rho}, the fact that, by (\ref{splitting-norm}),
$$\rho^2=\ga^2+\sum_{j=1}^k \alpha_j^2,$$
and the splitting lemma, we have
\[
m_{V,\rho}\ge F(v)+\sum_{j=1}^k m_{\alpha_j}\ge F(v)+m_{{\alpha}}\ge F(v)+m_\rho
\]
where, for easier notation, we write $\alpha = \max_j \alpha_j$. Moreover,
using the equation for $v$, it is easy to check that 
\[
I_{\lambda}(v) = \max_{t>0}  I_{\lambda}(tv).
\]

Let $\beta>0$ be such that
\[
\lambda_\beta = \lambda,
\]
according to \eqref{eq:lambda_rho}. 
By Appendix \ref{app:a}, $Z_\beta$ satisfies the limit equation with multiplier 
$\lambda$, and 
$\beta\le \alpha$ (actually, it is possible to show that $\beta = \alpha$ and $w = Z_\beta$, even though we do not need to use this). We also write 
\begin{equation}\label{eq:fincontr}
\theta = \min\left(1,\frac{2}{N}\right),
\end{equation}
so that assumption \eqref{eq:newass} reads $\|V\|_\infty\rho^2 < 2\theta m_\rho$. Using repeatedly Appendix \ref{app:a} and the above arguments, we have
\[
\begin{split}
(1+\theta)m_\rho &> m_\rho + \frac12\|V\|_\infty\rho^2 \ge  m_{V,\rho} \ge m_\alpha + F(v) = m_\alpha + I_{\lambda}(v) - 
\frac{\lambda}{2}\gamma^2 = m_\alpha + \max_{t>0} I_{\lambda}(tv) - 
\frac{\lambda}{2}\gamma^2\\
&\ge m_\alpha + \max_{t>0} I_{\infty,\lambda}(tv) - 
\frac{\lambda}{2}(\rho^2-\alpha^2) \ge m_\alpha + I_{\infty,\lambda}(Z_\beta) - 
\frac{\lambda}{2}(\rho^2-\alpha^2) \\
&\ge m_\alpha + m_\beta + 
\frac{\lambda}{2}(\alpha^2 + \beta^2 - \rho^2). 
\end{split}
\]
Since $m_\beta\ge m_\alpha > m_\rho$ and $\theta\le1$, we deduce that 
\[
\alpha^2 + \beta^2 - \rho^2<0.
\] 
Using Proposition \ref{prop-bounded-PS} and \eqref{eq:poho} we obtain
\[
(1+\theta)m_\rho > \frac{\rho^q}{\alpha^q}m_\rho + \frac{\rho^q}{\beta^q}m_\rho - 
\frac{\rho^2-\alpha^2 - \beta^2}{\rho^2} \left(q + \frac{4\theta}{N(p-2)-4} \right)\,m_\rho,
\]
that is
\begin{equation}\label{eq:penultima}
1+q+\frac{N(p-2)\theta}{N(p-2)-4}> \frac{\rho^q}{\alpha^q} + \frac{\rho^q}{\beta^q} +
\left(q + \frac{4\theta}{N(p-2)-4} \right)
\left( \frac{\alpha^2}{\rho^2}+\frac{\beta^2}{\rho^2}\right) .
\end{equation}
On the other hand, by elementary arguments we have that
\begin{multline*}
\min\left\{x^{-q/2} + y^{-q/2} + A(x + y) : x,y>0,\ x+y \le 1 \right\}\\
=
2\min\left\{x^{-q/2} + Ax : x>0,\ x\le \frac12 \right\}
=
\begin{cases}
(2 + q)\left(\dfrac{2A}{q}\right)^{q/(2 + q)}  & \text{if }\left(\dfrac{q}{2A}\right)^{2/(2 + q)} \le \dfrac12,\smallskip\\
2^{(q+2)/2} + A  & \text{if }\left(\dfrac{q}{2A}\right)^{2/(2 + q)} \ge \dfrac12.
\end{cases}
\end{multline*}
Now, 
\[
\left(\dfrac{q}{2A}\right)^{2/(2 + q)} \le \dfrac12
\qquad\implies\qquad
(2 + q)\left(\dfrac{2A}{q}\right)^{q/(2 + q)} \ge (2+q)2^{q/2},
\]
and \eqref{eq:penultima} becomes
\[
1+q+\frac{N(p-2)\theta}{N(p-2)-4} > (2+q)2^{q/2}
\]
and finally
\[
\theta > \frac{N(p-2)-4}{N(p-2)}\left((2+q)2^{q/2}-1-q\right)
\]
Using the elementary inequality
\[
(2+q)2^{q/2}-1-q > q\ln2  + 1 >\frac12 q+ 1 \qquad \text{for every }q>0
\]
we obtain
\[
\theta > \frac{N(p-2)-4}{N(p-2)}\cdot \frac{2(p-2)}{N(p-2)-4} = \frac{2}{N},
\]
in contradiction with \eqref{eq:fincontr}.

%
%
%
%
%

\section{The proof of Theorem \ref{thm:main} in case $(V_2)$}\label{sec-pole}

The strategy is similar to the one used in the case of $(V_1)$, in particular we use the same linking structure as in Section \ref{sec-rad}. First we observe that the results from Section \ref{sec-rad} up to and including Proposition \ref{prop-MP-geom} still hold true without any changes in the proof. The following result provides an upper bound for $m_{V,\rho}$.

\begin{lemma}\label{lem:stimabaseconpolo}
If
$$\|V\|_{\frac{N}{2}}<\frac{2N(p-2)}{N(p-2)-4}\left(\left(\frac{N+2}{N}\right)^{\frac{N(p-2)-4}{N(p-2)}}-1\right)\frac{m_\rho}{\|Z_\rho\|_{2^\star}^2}$$
then $m_{V,\rho}<\left(\frac{N+2}{N}\right)m_\rho$.
\end{lemma}

\begin{proof}
We have
\begin{equation}\label{eq:Nge3}
\begin{aligned}
m_{V,\rho}&\le \max_{h\in\R,\ y\in\R^N} F(h\star Z_\rho(\cdot-y))=\max_{h\in\R,\ y\in\R^N} \left(F_\infty(h\star Z_\rho)+\frac{1}{2}\int_{\R^N}V(x+y)(h\star Z_\rho)^2 dx\right)\\
  &\le\max_{h\in\R} \left(F_\infty(h\star Z_\rho)+\frac{1}{2}\|V\|_{\frac{N}{2}}\|h\star Z_\rho\|_{2^\star}^2\right)\\
&=\max_{h\in\R}\left(\frac{e^{2h}}{2}(\|\nabla Z_\rho\|_2^2+\|V\|_{\frac{N}{2}}\|Z_\rho\|_{2^\star}^2)-\frac{e^{\frac{N}{2}h(p-2)}}{p}\|Z_\rho\|_{p}^{p}\right).
\end{aligned}
\end{equation}
Differentiating with respect to $h$ and using the fact that
$$m_\rho=\frac{N(p-2)-4}{4p}\|Z_\rho\|_{p}^{p}$$
we see that
$$m_{V,\rho}\le m_\rho e^{\frac{N}{2}(p-2)h_0},$$
where $h_0$ is such that
$$e^{2h_0}(\|\nabla Z_\rho\|_2^2+\|V\|_{\frac{N}{2}}\|Z_\rho\|_{2^\star}^2)-\frac{N(p-2)}{2p}e^{\frac{N}{2}(p-2)h_0}\|Z_\rho\|_{p}^{p}=0.$$
Therefore  $m_{V,\rho}<(1+\theta)m_\rho$ if
$$e^{\frac{N}{2}(p-2)h_0}<1+\theta$$
or equivalently
$$\left(\frac{\|\nabla Z_\rho\|_2^2+\|V\|_{\frac{N}{2}}\|Z_\rho\|_{2^\star}^2}{m_\rho}\frac{N(p-2)-4}{2N(p-2)}\right)^{\frac{N(p-2)}{N(p-2)-4}}<1+\theta,$$
which is the case if
$$\|V\|_{\frac{N}{2}}<\frac{2N(p-2)}{N(p-2)-4}((1+\theta)^{\frac{N(p-2)-4}{N(p-2)}}-1)\frac{m_\rho}{\|Z_\rho\|_{2^\star}^2},$$
because 
\[
\frac{\|\nabla Z_\rho\|_2^2}{m_\rho} = \frac{2N(p-2)}{N(p-2)-4}.
\]
\end{proof}

Propositions \ref{prop-boundary}, \ref{prop:bddPSseq} also hold true, but their proofs require some adaptation.

\begin{proof}[Proof of Proposition \ref{prop-boundary} if $(V_2)$ holds]
Since for all $(y,h)\in\R^N\times\R$
\begin{eqnarray*}
F(h\star Z_\rho(\cdotp-y))&=&F_\infty(h\star Z_\rho)+\int_{\R^N}V(x)h\star Z_\rho^2 (x-y)dx\\
&=&
\frac{e^{2h}}{2}\int_{\R^N}|\nabla Z_\rho(x)|^2dx-\frac{e^{{N\over 2}(p-2)h}}{p} \int_{\R^N}Z_\rho^{p}(x)dx+ \int_{\R^N}V(x)(h\star Z_\rho)^2(x-y)dx,
\end{eqnarray*}
and taking into account
$$
0\le \int_{\R^N}V(x)(h\star Z_\rho)^2(x-y) dx\le \|V\|_{N/2}\, \|h\star Z_\rho\|_{2^\star}^2=Ce^{2h},
$$
it is readily seen that
$$
\lim_{h\to-\infty}F(h\star Z_\rho(\cdotp-y))=0\quad\mbox{ and }\quad
\lim_{h\to+\infty}F(h\star Z_\rho(\cdotp-y))=-\infty\quad\mbox{uniformly in }y\in\R^N.
$$
Therefore we can fix $h_1<0<h_2$ so that
\beq\label{eq1}
\sup_{(y,h)\in\R^N\times\{h_1,h_2\}}F(h\star Z_\rho(\cdotp-y))\le {m\over 2}.
\eeq
Now, observe that
\[
\begin{aligned}
  &\max_{h\in[h_1,h_2]}\int_{\R^N}V(x)(h\star Z_\rho)^2(x-y)dx\\
  &\hspace{1cm}
    = \max_{h\in[h_1,h_2]}\left(\int_{B_{|y|\over 2}(0)}V(x)(h\star Z_\rho)^2(x-y)dx+ \int_{\R^N\setminus B_{|y|/ 2}(0)}V(x)(h\star Z_\rho)^2(x-y)dx\right)\\
  &\hspace{1cm}
    \le \|V\|_{N\over 2}  \max_{h\in[h_1,h_2]}\|h\star Z_\rho\|^2_{L^{2^*}(\R^N\setminus B_{|y|/ 2}(0))}
         + \|V\|_{L^{\frac{N}{2}}(\R^N\setminus B_{|y|/ 2}(0))}  \max_{h\in[h_1,h_2]}\|h\star Z_\rho\|^2_{2^*}.
\end{aligned}
\]
hence
\[
\begin{aligned}
  &\limsup_{|y|\to+\infty}\max_{y\in[h_1,h_2]}F(h\star Z_\rho(\cdotp-y))\\
  &\hspace{1cm}
    =\limsup_{|y|\to+\infty}\max_{y\in[h_1,h_2]}\left(F_\infty(h\star Z_\rho)+\frac{1}{2}\int_{\R^N}V(x)(h\star Z_\rho)^2(x-y)dx\right)\le m_\rho
\end{aligned}
\]
which, together with $F(Z_\rho(\cdotp-y))>m_\rho$ for all $y\in\R^N$, yields
\beq\label{eq2}
\lim_{|y|\to\infty}\max_{h\in[h_1,h_2]}F(h\star Z_\rho(\cdotp-y))=m_\rho.
\eeq
Now Proposition \ref{prop-boundary} follows from \eqref{eq1} and \eqref{eq2}.
\end{proof}

%
%
\begin{proof}[Proof of Proposition \ref{prop:bddPSseq} if $(V_2)$ holds]
The existence of a Palais-Smale sequence follows as in Section \ref{sec-rad}. 
It remains to prove that the sequence is bounded, and that estimate \eqref{eq:lafromab} holds true. Using the notation from \eqref{eq:def-a_n} we observe that
\begin{equation}\label{eq:Nge3bis}
c_n=\int_{\R^N} V(x)v_n^2 dx\le\|V\|_{\frac{N}{2}}\|v_n\|_{2^\star}^2\le A^2\|V\|_{\frac{N}{2}}a_n,
\end{equation}
where $A$ is the Aubin-Talenti constant as in $(V_2)$. Concerning $d_n$, we observe that, 
by the Gagliardo-Nirenberg-Sobolev inequality,
\begin{equation}\label{eq:Nge3tris}
|d_n|\le \|W\|_N\|v_n\|_{2^\star}\|\nabla v_n\|_2\le A\|W\|_N\|\nabla v_n\|_2^2
\end{equation}
so that using \eqref{a_n-bd}, we have
\begin{equation}
\label{bound-a_n-pole}
a_n(N(p-2)-4-2NA^2\|V\|_{\frac{N}{2}}-4\|W\|_N A)\le 4N(p-2)m_\rho.
\end{equation}
This implies that $a_n$ is bounded thanks to $(V_2)$, which guarantees that
$$N(p-2)-4-2NA^2\|V\|_{\frac{N}{2}}-4\|W\|_N A>0.$$
In order to show that $\la_n$ admits a subsequence which converges to $\la>0$, we can argue as in the proof of Proposition \ref{prop-bounded-PS} and see that, up to a subsequence, $\la_n\to\la\in\R$ satisfying \eqref{lambda>0}, so that $\la>0$ provided
$$(p-2)|d|<m_\rho(2N-p(N-2)).$$
Using (\ref{bound-a_n-pole}), it is possible to see that
\[
\begin{aligned}
(p-2)|d|&\le (p-2)\|W\|_N A a\\
&\le\|W\|_N \frac{4AN(p-2)^2 m_\rho}{N(p-2)-4-2N A^2\|V\|_{\frac{N}{2}}-4A\|W\|_N}\\
&\le m_\rho(2N-p(N-2)),
\end{aligned}
\]
if $(V_2)$ holds. Finally, \eqref{eq:Nge3bis} yields
\[
\begin{aligned}
c&\le \|V\|_{\frac{N}{2}}a \\
&\le\|V\|_{\frac{N}{2}}  \frac{4NA^2(p-2) m_\rho}{N(p-2)-4-2N A^2\|V\|_{\frac{N}{2}}-4A\|W\|_N}\\
&\le\frac{4}{N}m_\rho,
\end{aligned}
\]
if $(V_2)$ holds. Then Lemma \ref{lem:stimabaseconpolo} and \eqref{lambda>0} provide \eqref{eq:lafromab} also in this case (recall that here we are assuming $N\ge3$, so that $\bruttateta=\frac{2}{N}$). The proposition follows.
\end{proof}

The conclusion of the proof of Theorem \ref{thm:main} under assumption $(V_2)$ is exactly the same as in the case of assumption $(V_1)$.

\begin{remark}\label{rem:N<3inassV2}
In case $N=1,2$ we have that $2^*=\infty$, and the above estimates can be modified in a straightforward way (although with much heavier calculations). Let $p\le q <\infty$ (or even $q=\infty$ in dimension $N=1$). Then in \eqref{eq:Nge3} we can use the fact that
\[
\int_{\R^N}V(x+y)(h\star Z_\rho)^2 \le \|V\|_{\frac{q}{q-2}} \|h\star Z_\rho\|^2_{q}
= e^{{N\over q}(q-2)h} \|V\|_{\frac{q}{q-2}} \|Z_\rho\|^2_{q},
\]
while \eqref{eq:Nge3bis}, \eqref{eq:Nge3tris} become
\[
c_n\le\|V\|_{\frac{q}{q-2}}\|v_n\|_{q}^2,
\qquad
|d_n|\le \|W\|_{\frac{2q}{q-2}}\|v_n\|_{q}\|\nabla v_n\|_2,
\]
respectively, so that both terms can be estimated in terms of $a_n$ and $\rho$ by using the Gagliardo-Nirenberg inequality
\[
\|v\|_q \le C_{N,q} \|v\|_2^{1-\frac{N(q-2)}{2q}}\|\nabla v\|_2^{\frac{N(q-2)}{2q}}.
\]
As a consequence, an assumption analogous to $(V_2)$, holding for any $N\ge1$, could be written in terms of 
$\|V\|_{\frac{q}{q-2}}$, $\|W\|_{\frac{2q}{q-2}}$. 
\end{remark}

\section{Proof of Theorem \ref{thm:rad}}\label{sec:rad}
In the radial case we can apply a mountain pass argument as in \cite{J}. We set $S^r_\rho:=S_\rho\cap H^1_{rad}(\R^N)$ and
$$\Si_\rho^r:=\{\si\in\cC([0,1],S^r_\rho):\si(0)=h_1\star Z_\rho,\si(1)=h_2\star Z_\rho\}$$
with $h_1<0<h_2$ such that
\begin{equation}
\max\{F(h_1\star Z_\rho),\, F(h_2\star Z_\rho)\}<m_\rho.
\end{equation}
The choice is possible because
\[
  F(h\star Z_\rho) \to \begin{cases} 0&\qquad\text{as $h\to-\infty$}\\ -\infty&\qquad\text{as $h\to\infty$.} \end{cases}
\]
We look for a critical point of $F$ at the level
$$m_{V,\rho}^r:=\inf_{\si\in\Si_\rho^r}\max_{t\in[0,1]} F(\si(t)).$$

\begin{lemma}
Under the assumptions of Theorem \ref{thm:rad}, we have $m_{V,\rho}^r>m_\rho$.
\end{lemma}

\begin{proof}
Using the fact that $F_\infty(u^\star)\le F_\infty(u)$ if $u^\star$ is the radially decreasing rearrangement of $u\in H^1(\R^N)$, we have
$$m_\rho=m_\rho^r:=\inf_{\si\in\Si_\rho^r}\max_{t\in[0,1]}F_\infty(\si(t)),$$
so that it is enough to prove $m_{V,\rho}^r>m_\rho^r$. Clearly $m_{V,\rho}^r\ge m_\rho^r$ because $V\ge 0$. If we assume equality, then there would exist a sequence $\si_n\in\Si_\rho^r$ such that
$$0\le\max_{t\in[0,1]}F(\si_n(t))-m^r_\rho\le\frac{1}{n}.$$
As a consequence, using Proposition \ref{prop-Ekeland}, it would be possible to construct a Palais-Smale sequence $v_n\in H^1_r(\R^N)$ such that $v_n\to Z_\rho$ strongly in $H^1(\R^N)$ and
$$\dist(v_n,\si_n([0,1]))=\|v_n-\si_n(\bar{t}_n)\|\to 0,$$
for some $\bar{t}_n\in[0,1]$. Therefore
\begin{equation}\notag
\begin{aligned}
\max_{t\in[0,1]}F(\si_n(t))\ge F(\si_n(\bar{t}_n))
  &= F(v_n)+o(1) = F_\infty(v_n)+\frac{1}{2}\int_{\R^N}V v_n^2 dx+o(1)\\
  &\to m_\rho+\frac{1}{2}\int_{\R^N}V Z_\rho^2 dx,
\end{aligned}
\end{equation}
a contradiction.
\end{proof}

The proof of Proposition \ref{prop-MP-geom} holds if $V$ satisfies $(V^{rad}_1)$ or $(V^{rad}_2)$ because we did not use \eqref{eq:newass}, \eqref{eq:newass2}. Therefore we obtain a bounded sequence $v_n\in H^1_r(\R^N)$ such that
\[
F(v_n)\to m_{V,\rho}^r \qquad\text{and}\qquad \nabla_{S_\rho\cap H^1_{rad}(\R^N)} F (v_n)\to 0.
\]
Moreover, the Lagrange multipliers
$$\la_n:=-\frac{DF(v_n)[v_n]}{\rho^2}$$
admit a subsequence converging to some $\la>0$. (In the proof of Proposition \ref{prop:bddPSseq}, the assumptions \eqref{eq:newass}, \eqref{eq:newass2} were used only to prove the bound from above of $\lambda$, which was crucial in the estimates after the splitting lemma but this is not needed here.)

As a consequence there exists a subsequence of $v_n$ converging weakly in $H^1(\R^N)$ to some weak solution $v\in H^1(\R^N)$ of
$$-\Delta v+(V+\la)v=|v|^{p-2}v \qquad\text{in $H^1_{rad}(\R^N)$}$$
with $\|v\|_2\le \rho^2$. The compactness of the Sobolev embedding $H^1_{rad}(\R^N)\hookrightarrow L^{p}(\R^N)$ for $p\in(2,2^*)$ implies $v_n\to v$ in $L^{p}(\R^N)$ strongly. As a consequence, using $\la>0$, we have
$$v_n:=(-\Delta+\la+V)^{-1}(|v_n|^{p-2}v_n+(\la-\la_n)v_n)\to v$$
in $H^1(\R^N)$ strongly, so that $\|v\|_2=\rho^2$.

\appendix

\section{Relations for solutions of the autonomous equation.}\label{app:a}

Let $a>0$ and $w\in H^1(\R^N)$ solve 
\[
-\Delta w+ a w=|w|^{p-2}w.
\] 
We collect here for the reader's convenience some well known relations that we use throughout the paper, starting from equations \eqref{eq:lambda_rho} and \eqref{eq:poho}:
\begin{itemize}
\item we have
\[
\begin{split}
\|\nabla w\|_2^2 + a \|w\|_2^2 = \|w\|_p^p,\qquad
\frac{N-2}{2}\|\nabla w\|_2^2 + \frac{N}{2} a \|w\|_2^2 = \frac{N}{p} \|w\|_p^p
\end{split}
\]
(the second one is the Pohozaev identity). This yields
\[
a\|w\|_2^2 = \frac{2N - p(N-2)}{p-2} I_{\infty,a}(w) = \frac{4N - 2p(N-2)}{N(p-2)-4} F_{\infty}(w).
\]
\item let $\rho$ be such that $\lambda_\rho = a$, and $Z_\rho$ be defined as usual. Then
\[
-\Delta Z_\rho + a Z_\rho=Z_\rho^{p-1}
\]
as well, and it satisfies the same relations than $w$. Moreover, for every $u\in H^1(\R^N)$, $u\not \equiv0$,
\[
I_{\infty,a}(Z_\rho) \le \max_{t>0}  I_{\infty,a}(tu)
\]
Since $\max_{t>0}  I_{\infty,a}(tw) = I_{\infty,a}(w)$, we deduce
\[
I_{\infty,a}(w) \ge I_{\infty,a}(Z_\rho),\qquad
\|w\|_2 \ge \|Z_\rho\|_2,\qquad
F_{\infty}(w) \ge F_{\infty}(Z_\rho).
\]
\end{itemize}


\end{document}